\documentclass[12pt]{article}
\usepackage{amsmath, amsthm}
\usepackage[top=4cm,left=3cm,right=3cm,bottom=4cm,headheight=13.6pt]{geometry}\usepackage{amsfonts}
\usepackage{amssymb}
\usepackage{fancyhdr}
\usepackage{enumerate}
\usepackage{amsmath,amsthm}
\usepackage[dvips]{graphicx}
\usepackage{color}
\usepackage{eepic}
\usepackage{epic}
\usepackage[dvips]{rotating}
\usepackage{sectsty}
\usepackage{wrapfig}

\theoremstyle{plain}
\newtheorem{theorem}{Theorem}[section]
\newtheorem{lemma}[theorem]{Lemma}

\theoremstyle{definition}

\newtheorem{remark}[theorem]{Remark}

\theoremstyle{remark}

 \numberwithin{equation}{section} 

\begin{document}
\title{Blow-up Rate Estimates for a Semilinear Heat Equation with a Gradient Term}
\author{Maan A. Rasheed and Miroslav Chlebik}
\maketitle

\abstract 
We consider the the pointwise estimates and the blow-up rate estimates for the zero Dirchilet problem of the semilinear heat equation with a gradient term $u_t= \Delta u-|\nabla u|^2+ e^u,$  which has been considered  by J. Bebernes and D. Eberly in  \cite{48}.

\section{Introduction} 
Consider the following initial-boundary value problem
\begin{equation}\label{1e} \left.\begin{array}{ll}
u_t= \Delta u-h(|\nabla u|) + f(u),&\quad (x,t) \in B_R  \times (0,T),\\
u(x,t)=0,& \quad (x,t)\in \partial  B_R  \times (0,T),\\
u(x,0)=u_0(x),&\quad x \in B_R,\end{array}\right\} \end{equation}
 where $f \in C^1(R),$ $h\in C^{1}([0,\infty)),$ $f,h >0,~h^{'}\ge 0$ in  $(0,\infty),~f(0)\ge 0, h(0)=h^{'}(0)=0,$ \begin{equation}\label{viva}|h(\xi)|\le O(|\xi|^2),\end{equation}  \begin{equation} \label{m2} sh^{'}(s)-h(s)\le Ks^q,\quad \mbox{for} \quad s>0,~ 0\le K < \infty,~ q>1,\end{equation}     
 $u_0\ge 0$ is smooth, radial nonincreasing function, vanishing on $\partial B_R,$ this means it satisfies the following conditions
\begin{equation}\label{3e}\left. \begin{array}{ll}
 u(x)=u_0(|x|),& \quad x \in B_R,\\  u_0(x)=0,&\quad x \in \partial B_R,\\
u_{0r}(|x|)\le 0,& \quad  x \in {B}_R.\end{array} \right\}\end{equation} Moreover, we assume that 
\begin{equation}\label{ee3} \Delta u_0+ f(u_0) - h(|\nabla u_0|) \ge 0, \quad  x \in B_R.\end{equation} 
The special case 
\begin{equation}\label{z1} u_t= \Delta u- |\nabla u|^q + u|u|^{p-1}, ~p,q>1\end{equation} was introduced in \cite{24}  and it was studied and discussed later by many authors see for instance \cite{25,26}. The main issue in those works was to determine for which $p$ and $q$ blow-up in finite time (in the $L^\infty$-norm) may occur. It is well known  that it occurs if and only if $p>q$ (see \cite{25}). Equation (\ref{z1}) in $R^n$ was considered  from similar point of view, in this case blow-up in finite time is also known to occur when $p>q,$ but unbounded global solutions always exist (see \cite{26}). For bounded domains, it has been shown in  \cite{29} for equation (\ref{z1}) with general convex domain $\Omega$ that, the blow-up set is compact. Moreover if $\Omega=B_R,$ then $x=0$ is the only possible blow-up point and the upper pointwise rate estimate takes the following form
$$u \le c{|x|}^{-\alpha },\quad (x,t) \in B_R\setminus\{0\} \times [0,T),$$
 for any $\alpha >2/(p-1)$ if $q\in (1,2p/(p+1) ),$ and for  $\alpha >q/(p-q)$ if  $q \in [2p/(p+1),p).$ 
We observe that $q/(p-q)>2/(p-1)$ for $q>2p/(p+1),$ therefore, the blow-up profile of solutions of  equation (\ref{z1}) is similar to that of $u_t=\Delta u +u^p$ as long as 
$q<2p/(p+1)$ (see \cite{1}), whereas for $q$ grater that this critical value, the gradient term induces an imprtant effect on the profile, which becomes more singular. 

On the other hand, it was  proved  in  \cite{27,29,30,28} that the upper (lower) blow-up rate estimate in terms of the blow-up time $T$ in the case $q<2p/(p+1)$ and $u\ge 0,$ takes the following form
$$c(T-t)^{-1/(p-1)}\le u(x,t)\le C(T-t)^{-1/(p-1)}. $$

J. Bebernes and D. Eberly  have considered  in \cite{48} a second special case of (\ref{1e}), where $f(s)=e^s, h(\xi)=\xi^2,$ namely
\begin{equation}\label{e1}\left. \begin{array}{ll}
u_t= \Delta u-|\nabla u|^2+ e^u,&\quad (x,t)\in B_R \times (0,T),\\
u(x,t)=0,&\quad (x,t) \in \partial B_R \times (0,T),\\
u(x,0)=u_0(x),&\quad x \in B_R. \end{array}\right\} \end{equation}
 The semilinear equation in (\ref{e1}) can be viewed as the limiting case of the critical splitting as $p\rightarrow \infty $ in the  equation (\ref{z1}). It has been proved that, the solution of the above problem with $u_0$ satisfies (\ref{3e}) may blow up in finite time and the only possible blow-up point is $x=0.$ Moreover, if we consider the problem in any general bounded domain $\Omega$ such that $\partial\Omega$ is analytic, then the bow up set is a compact set. On the other hand, they proved that, if $x_0$ is a blow-up point for problem (\ref{e1}) with the finite blow-up time $T$; then
 $$\lim_{t\rightarrow T^-}[u(x_0,t)+m \log (T-t) ]= k,$$ for some $m \in Z^+$ and for some $k \in R.$ The analysis therein is based on the observation that the transformation $v=1-e^{-u}$ changes the first equation in problem (\ref{e1}) into the linear equation $v_t=\Delta v+1,$ moreover, $x_0$ is a blow-up point for (\ref{e1}) with blow-up time $T$ if and only if $v(x_0,T)=1.$
    
    In this paper we consider problem (\ref{e1}) with (\ref{3e}), our aim is to derive the upper pointwise estimate for the classical solutions of this problem and to find a formula for the upper (lower) blow-up rate estimate.
 
  \section{Preliminaries} 
 The local existence and uniqueness of classical solutions to problem (\ref{1e}), (\ref{3e}) is well known by \cite{23,37}. Moreover, the gradient function $\nabla u$ is bounded as long as the solution $u$ is bounded due to (\ref{viva}) (see \cite{2}).

 The following lemma shows some properties of the classical solutions of problem (\ref{1e}) with (\ref{3e}). We may denote for simplicity $u(r,t)=u(x,t).$
 
\begin{lemma}\label{No}
 Let $u$ be a classical solution to the problem classical solution of problem (\ref{1e}) with (\ref{3e}). Then
  
  \begin{enumerate}[\rm(i)]   \item $u>0$ and it is radial nonincreasing in $B_R\times (0,T).$ Moreover if $u_0 \not\equiv 0,$ then $u_r <0 $ in $(0,R] \times (0,T).$ 
\item $u_t\ge 0$ in $\overline B_R\times [0,T).$
 \end{enumerate}
 \end{lemma}
 
Depending on Lemma \ref{No}, the problem (\ref{1e}) with (\ref{3e}) can be rewritten as follows  
  \begin{equation}\label{8e}\left. \begin{array}{ll}
u_t=  u_{rr}+\frac{n-1}{r}u_r- h(-u_r) + f(u),&\quad (r,t) \in (0,R)  \times (0,T),\\
u_r(0,t)=0,\quad u(R,t)=0,&\quad  t \in  [0,T),\\
u(r,0)=u_0(r),& \quad r \in [0,R],\\
u_r(r,t)<0,\quad & \quad(r,t)\in (0,R] \times (0,T).\end{array}\right\} \end{equation}

\section{Pointwise Estimate}
  
Inorder to derive a formula to the pointwise estimate for problem (\ref{8e}), we need first to recall the following theorem, which has been proved in \cite{29}.  
\begin{theorem}\label{maA}
Assume that, there exist two functions $F \in C^2([0,\infty))$ and $c_\varepsilon \in C^2([0,R]),\varepsilon>0,$ such that \begin{equation}\label{01} c_\varepsilon(0)=0, c_\varepsilon ^{'}\ge 0, \quad F>0, F^{'},F^{''} \ge 0,\quad \mbox{in} \quad (0,\infty), \end{equation}
\begin{equation}\label{02} f^{'}F-fF^{'}-2c_\varepsilon ^{'}F^{'}F +c_\varepsilon ^2F^{''}F^2-2^{q-1}K c_\varepsilon ^qF^qF^{'}+AF \ge 0, \quad u > 0, 0<r < R, \end{equation}
where $$A=\frac{c_\varepsilon^{''}}{c_\varepsilon}+\frac{n-1}{r}\frac{c_\varepsilon ^{'}}{c_\varepsilon}-\frac{n-1}{r^2},$$ $\frac{ c_\varepsilon (r)}{r} \rightarrow 0$ uniformly on $[0,R]$ as 
$\varepsilon \rightarrow 0,$ and 
$$ G(s)=\int _s^\infty \frac{du}{F(u)} <\infty,  \quad s>0.$$ Let $u$ is a blow-up solution to problem (\ref{8e}), where $u_0$ satisfies 
\begin{equation}\label{13M}u_{0r}\le -\delta,\quad r\in (0,R],\quad \delta>0.\end{equation}  
Suppose that, $T$ is the blow-up time. Then the point $r=0$ is the only blow-up point, and there is $\varepsilon _1 >0$ such that \begin{equation}\label{G1} u(r,t)\le G^{-1}(\int _0^r c_{\varepsilon _1}(z)dz), \quad (r,t) \in (0,R] \times (0,T).\end{equation}
\end{theorem}

We are ready now to drive a formula to the pointwise estimate for the blow-up solutions of problem (\ref{e1}) with (\ref{3e}).
\begin{theorem}\label{Mac}
Let $u$ be a blow-up solution to problem (\ref{e1}), assume that $u_0$ satisfies (\ref{3e}) and (\ref{13M}).Then the upper pointwise estimate takes the following form
$$u(x,t)\le \frac{1}{2\alpha}[\log C- m \log(r)], \quad (r,t) \in (0,R] \times (0,T),$$  where  $\alpha \in (0,1/2], C>0, m >2.$
\end{theorem}
\begin{proof}
Let $c_\varepsilon =\varepsilon r^{1+\delta},$ where $\delta \in (0,\infty).$  

It is clear that $c_\varepsilon$ satisfies the assumptions (\ref{01}) in Theorem \ref{maA}, so that (\ref{02}) becomes
\begin{eqnarray}\label{03} &&f^{'}F-fF^{'}-2\varepsilon (1+\delta)r^\delta F^{'}F +\varepsilon ^2r^{2+2\delta}F^{''}F^2 \nonumber\\ &&-2^{q-1}K \varepsilon ^q r^{q+\delta q}F^qF^{'}+\frac{\delta(n+\delta)}{r^2}F\ge 0 ,~u> 0, 0<r < R. \end{eqnarray}
 For the semilinear equation in (\ref{e1}) it is clear that $K\ge1, q=2.$ To make use of Theorem \ref{maA} for problem (\ref{e1}), assume that $$F(u)=e^{2\alpha u},\quad \alpha \in (0,1/2].$$ It is clear that $F$ satisfies all the assumptions (\ref{01}) in Theorem \ref{maA}. With this choice of  $F$ the inequality (\ref{03}) takes the form 
\begin{align*}
(1-2\alpha)e^{(1+2\alpha)u}+4\alpha ^2 \varepsilon ^2r^{2(1+\delta)}e^{6\alpha u}+\frac{\delta(n+\delta)}{r^2}e^{2\alpha u}\ge \\ 4\alpha \varepsilon (1+\delta)r^\delta e^{4\alpha u}+4 \alpha \varepsilon ^2 r^{2(1+\delta)}e^{6\alpha u}, \quad u\ge 0 ,0<r\le R \end{align*}
provided $\alpha \le\frac{1}{2+4\varepsilon R^{\delta}(1+\delta)}.$

Define the function $G$ as in Theorem \ref{maA} as follows $$ G(s)=\int _s^\infty \frac{du}{e^{2\alpha u}}=\frac{1}{2\alpha e^{\alpha s}}, \quad s>0.$$ Clearly, $$G^{-1}(s)=-\frac{1}{2\alpha}\log (2\alpha s), \quad s>0.$$ 
Thus (\ref{G1}) becomes $$u(r,t)\le \frac{1}{2\alpha}[\log C-m \log (r)],\quad (r,t) \in (0,R] \times (0,T),$$  where
$C=\frac{2+\delta}{2\varepsilon\alpha}, \quad m=2+\delta.$
\end{proof}  
\begin{remark} Theorem \ref{Mac} shows that, with choosing $\alpha=1/2,$ the upper pointwise estimate for problem (\ref{e1}) is the same as that for $u_t=\Delta u+e^u,$ which has been considered in \cite{1}. Therefore, the gradient term in problem (\ref{e1}) has no effect on the pointwise estimate.   
\end{remark}

\section{Blow-up Rate Estimate }
Since under the assumptions of Theorem \ref{Mac},  $r=0$ is the only blow-up point for the problem (\ref{e1}), therefore, in order to estimate the blow-up solution it suffices  to estimate only $u(0,t).$
The next theorem, which has been proved in \cite{29}, considers the upper blow-up rate estimate for the general problem (\ref{1e}).
\begin{theorem}\label{e}
  Let $u$ be a blow-up solution to problem (\ref{1e}), where $u_0\in C^2(\overline{B}_R)$ and satisfies (\ref{3e}), (\ref{ee3}). Assume that $T$ is the blow-up time and $x=0$ is the only possible blow-up point. If there exist a function, $F \in C^2([0,\infty))$ such that $F>0$ and $F^{'} ,F^{''}\ge 0$ in $(0,\infty),$ moreover, 
  \begin{equation}\label{poa}  f^{'}F-F^{'}f+F^{''}|\nabla u|^2 -F^{'}[h^{'}(|\nabla u|)|\nabla u|-h(|\nabla u|)] \ge 0,~ \mbox{in}~B_R\times (0,T),\end{equation} 
   then the upper blow rate estimate takes the from 
  $$u(0,t)\le G^{-1}(\delta(T-t)), \quad t \in (\tau,T),$$ where $\delta, \tau >0,$ $G(s)=\int_s^\infty \frac{du}{F(u)}.$ 
\end{theorem}

For problem (\ref{e1}), if one could choose a suitable function $F$ that satisfies the conditions, which have stated in Theorem \ref{e}, then the upper blow-up rate estimate for this problem would be held.
\begin{theorem}\label{X1}
Let $u$ be a blow-up solution to problem (\ref{e1}), where $u_0 \in C^2(\overline{B}_R)$ and satisfies (\ref{3e}), (\ref{13M}) and the monotonicity assumption $$\Delta u_0+ e^{u_0} -|\nabla u_0|^2 \ge 0,\quad x\in B_R,$$ suppose that $T$ is the blow-up time.Then there exist $C>0$ such that the upper blow-up rate estimate takes the following form
$$u(0,t)\le  \frac{1}{\alpha} [ \log C-\log(T-t)],\quad 0 <t <T,~\alpha \in (0,1].$$  
\end{theorem}
\begin{proof}
Let $$F(u)=e^{\alpha u},\quad \alpha \in (0,1].$$ It is clear that the inequality (\ref{poa}) becomes
$$(1-\alpha)e^{(1+\alpha)u}+\alpha ^2e^{\alpha u}|\nabla u|^2-\alpha e^{\alpha u}|\nabla u|^2\ge 0,$$
which holds for any $\alpha \in (0,1].$ 

Set $$G(s)=\int_s^\infty \frac{du}{e^{\alpha u}}=\frac{1}{\alpha e^{\alpha s}},\quad s>0.$$ Clearly, $$G^{-1}(s)=-\frac{1}{\alpha}\log (\alpha s), \quad s>0.$$ 
From Theorem \ref{e} there is $\delta >0$ such that $$u(0,t)\le  \frac{1}{\alpha} [\log (\frac{1}{\alpha \delta})-\log(T-t)],\quad \tau <t <T.$$
Therefore, there exist a positive constant, $C$ such that
$$u(0,t)\le  \frac{1}{\alpha} [\log C-\log(T-t)],\quad 0 <t <T.$$
\end{proof}

Next, we consider the lower blow-up rate for problem (\ref{e1}), which is much easier than the upper bound.

\begin{theorem}\label{X2}
 Let $u$ be a blow-up solution to problem (\ref{e1}), where $u_0$ satisfies (\ref{3e}) and (\ref{13M}). Suppose that $T$ is the blow-up time.Then there exist $c>0$ such that the lower blow-up rate estimate takes the following form
$$ \log c-\log(T-t) \le u(0,t), \quad 0 <t <T.$$  
\end{theorem}
\begin{proof}
Define $$U(t)=u(0,t),\quad t\in [0,T).$$
Since $u$ attains its maximum at $x=0,$ 
$$\Delta U(t) \le 0,\quad 0\le t<T.$$ 
From the semilinear equation in (\ref{e1}) and above, it follows that
\begin{equation}\label{xc} U_t(t)\le e^{U(t)}\le \lambda e^{U(t)},\quad 0< t<T,\end{equation} for $\lambda \ge1.$
 Integrate (\ref{xc}) from $t$ to $T,$ we obtain
 $$\frac{1}{\lambda (T-t)}\le e^{u(0,t)},\quad 0 <t <T.$$ It follows that
  $$ \log c-\log(T-t) \le u(0,t), \quad 0 <t <T,$$  where $c=1/\lambda.$
  \end{proof}
  \begin{remark}
  Theorem \ref{X2} (Theorem \ref{X1}, where $\alpha=1$) show that, the lower (upper) blow-up rate estimate for problem (\ref{e1}) is the same as for $u_t=\Delta u+e^u,$ which has been considered in \cite{1}, therefore, we conclude that, the gradient term in problem (\ref{e1}) has no effect on the blow-up rate estimate.
   \end{remark}

\end{document}